\documentclass[letterpaper,11pt]{article}

\usepackage{amsmath,amsfonts,amssymb,amsthm,mathrsfs,dsfont}
\usepackage{cases}
\usepackage{enumerate}
\usepackage[usenames,dvipsnames]{color}
\usepackage{verbatim} 
\usepackage{array} 
\usepackage{graphicx,epstopdf} 
\usepackage{mathtools}

\theoremstyle{plain} \newtheorem{thm}{Theorem}
\theoremstyle{plain} \newtheorem{lemma}[thm]{Lemma}
\theoremstyle{plain} 
\theoremstyle{plain} 
\theoremstyle{plain} 
\theoremstyle{plain} 
\theoremstyle{definition} 
\theoremstyle{definition} 
\theoremstyle{plain} 

\newcommand{\sub}[0]{\subseteq}
\newcommand{\sm}[0]{\setminus}
\renewcommand{\dots}[0]{,\ldots,}

\newcommand{\ov}[0]{\overline}

\newcommand{\beq}[1]{\begin{equation}\label{#1}}
\newcommand{\enq}[0]{\end{equation}}

\newcommand{\bn}[0]{\bigskip\noindent}
\newcommand{\mn}[0]{\medskip\noindent}

\newcommand{\E}[0]{{\mathbb{E}}}

\newcommand{\gG}[0]{\Gamma }

\newcommand{\gl}[0]{\lambda }

\newcommand{\go}[0]{\omega}
\newcommand{\gO}[0]{\Omega}

\newcommand{\A}[0]{{\cal A}}
\newcommand{\B}[0]{{\cal B}}

\newcommand{\f}[0]{{\cal F}}

\author{Jacob D. Baron\thanks{Department of Mathematics, Rutgers University, Piscataway, NJ. 
Supported by the U.S. Department of Homeland Security under Grant Award 2012-ST-104-000044. The views and conclusions contained in this document are those of the authors and should not be interpreted as necessarily representing the official policies, either express or implied, of the U.S. Department of Homeland Security.} \and Jeff Kahn\thanks{Department of Mathematics, Rutgers University, Piscataway, NJ. Supported by the National Science Foundation under Grant Awards DMS1201337 and DMS1501962.} }

\title{A Natural Extension of the BK Inequality}

\date{Sept 2016}

\begin{document}

\maketitle

\begin{abstract}
We extend the seminal van den Berg--Kesten Inequality \cite{BK} on disjoint occurrence of two events to a setting with arbitrarily many events, where the quantity of interest is the maximum number that occur disjointly. This provides a handy tool for bounding upper tail probabilities for event counts in a product probability space.
\end{abstract}

\section{Introduction}\label{Intro}

The purpose of this note is to prove a natural stochastic domination result that greatly extends a fundamental inequality on disjoint occurrence of events.

To begin we recall a few definitions. For (real-valued) random variables $X$ and $Y$,
$Y$ {\em stochastically dominates} $X$ (written $X \preccurlyeq Y$) if
$\Pr(Y \geq r) \geq \Pr(X \geq r) \; \forall \, r \in \mathbb{R}$. 
An event $A$ in a partially ordered $\gG$ is \emph{increasing} if its indicator is a
nondecreasing function, and \emph{decreasing} if its complement is increasing.
A probability measure $m$ on a partially ordered $\gG$ is \emph{positively associated} (PA) if
$m(A\cap B)\geq  m(A)m(B)$ whenever both
$A$ and $B\sub \gG$ are increasing (or, equivalently, whenever both are decreasing), and
note that any probability measure on a linearly ordered $\gG$ is PA. We write $[n]$ for $\{1,2,\ldots,n\}$.

Our setting is a finite product probability space
$(\Omega,\mu)=\prod_{i=1}^n (\Omega_i,\mu_i)$ with each $\gO_i$ partially ordered.
Events $A_1, A_2, \ldots, A_k$ ($\sub \gO$) are said to \emph{occur disjointly at} $\omega \in \Omega$
if there are disjoint $S_1\dots S_k\sub [n]$ such that for each $i \in [k]$ and
$\omega' \in \Omega$,  we have $\omega' \in A_i$ whenever $\omega'$ agrees with $\omega$ on $S_i$.
We write
\[ \square_{i=1}^k A_i = \{ \omega \in \Omega : A_1,\ldots,A_k \text{ occur disjointly at } \omega \}.\]
The study of disjoint occurrence was initiated by van den Berg and Kesten \cite{BK}, who
showed what is now called the ``BK Inequality'':
\begin{align}
\label{eq:BK}
\Pr(A\square B) \leq \Pr(A) \Pr(B)
\end{align}
for increasing $A, B \subseteq \{0,1\}^n$ (see also e.g.\ \cite[Section 2.3]{Grimmett}). 
The following (substantial) extension of this seminal result is apparently new \cite{vdB}.

\begin{thm}
\label{TBK'}
Let $(\Omega,\mu)=\prod_{i=1}^n (\Omega_i,\mu_i)$ be a finite product probability space
with the $\gO_i$'s partially ordered and the $\mu_i$'s PA. Given $A_1, A_2, \ldots, A_k\sub\Omega$,
let 
\[
X = \max\{ |I| : I \subseteq [k] \text{ and } \square_{i \in I}A_i \text{ occurs}\}.
\]
Let $Y_1\dots Y_k$ be independent Bernoullis with $\mathbb{E}Y_i=\Pr(A_i)$,
$Y = \sum Y_i$, and $\gl=\sum \E Y_i$.  If the $A_i$'s are all increasing, or all decreasing, then 
\begin{align}
\label{YstochdomX}
X \preccurlyeq Y.
\end{align}
%
\end{thm}

\bn \textbf{\emph{Remarks.}}
\begin{itemize}
\item[(i)] Taking $\Omega=\{0,1\}^n$, $k=2$ and $r=2$ in the definition of ``$X\preccurlyeq Y$" recovers \eqref{eq:BK} from \eqref{YstochdomX}.

\item[(ii)] The most spectacular of the developments growing out of \cite{BK} is Reimer's proof
\cite{Reimer} of the ``BK Conjecture'' (of \cite{BK}) which says that \eqref{eq:BK}
doesn't require that $A,B$ be increasing. In contrast, trivial
examples show this requirement (or some requirement) to be necessary in \eqref{YstochdomX};
for instance if $\Omega = \{0,1\}$ with uniform measure, $k=2$, $A_1 = \{0\}$
and $A_2=\{1\}$, then $\Pr(X \geq 1)  = 1 > 3/4 = \Pr(Y \geq 1).$

\item[(iii)] As a consequence of \eqref{YstochdomX}, the Chernoff Bound (e.g.\ \cite[Theorem 2.1]{JLR}) applied to $Y$ yields, for $t \geq 0$,
\begin{align}
\label{BKChernoff}
\Pr(X \geq \gl+t) \leq \exp\left[-\gl \, \varphi(t/\gl)\right] ~~~\left( \leq \exp\left[-t^2/(2(\gl+t/3))\right]\right)
\end{align}
(where $\varphi(x) = (1+x)\log(1+x)-x$ for $x > -1$, and $\varphi(-1)=1$). This looks similar to a lemma of Janson, proved (in slightly restricted form) in \cite[Lemma 2]{Janson} or \cite[Lemma 2.46]{JLR}:
\begin{lemma}
\label{JUB}
For events $A_1\dots A_k$
in a probability space, $\gl=\sum\Pr(A_i)$ and $t \geq 0$, letting 
\[\mbox{$Z = \max\{ |I| : I \subseteq [k]$, $\{A_i\}_{i \in I}$ are independent, and $\cap_{i \in I}A_i$ occurs$\}$,}\]
\begin{align}
\label{eq:JUB}
\Pr(Z \geq \gl+t) \leq \exp\left[-\gl \, \varphi(t/\gl)\right].
\end{align}
\end{lemma}

\hspace{.15in} But there are two big differences between \eqref{BKChernoff} and \eqref{eq:JUB}. On one hand, \eqref{eq:JUB} clearly applies more broadly. On the other hand, \eqref{BKChernoff} implies \eqref{eq:JUB} when it applies, since independent increasing (or decreasing) events, if they occur, necessarily occur disjointly (a standard observation easily extracted from the usual proof of Harris's Inequality \cite{Harris}). In fact when \eqref{BKChernoff} applies it can be much stronger than \eqref{eq:JUB}, because dependent events can easily occur disjointly---so $X$ can be much larger than $Z$, even though the bounds given for their upper tails are the same. For example, if $x_1,\ldots,x_k,y_1,\ldots,y_k$ are distinct vertices of the Erd\H{o}s--R\'{e}nyi random graph $G_{n,p}$ and, for $i \in [k]$, $A_i = \{\text{there is an } x_iy_i\text{-path}\}$, then $Z \leq 1$ but $X$ can be large.

\item[(iv)] It is not true that $Z \preccurlyeq Y$ in the generality of Lemma \ref{JUB}, as the example in Remark (ii) also shows.
\end{itemize}

\bigskip For \eqref{BKChernoff}, we can trade the requirement that the $A_i$'s be all increasing (or all decreasing) for the requirement that the $\Omega_i$'s be all linearly ordered:
\begin{thm}
\label{TBK2}
In the setting of Theorem \ref{TBK'}, with arbitrary $A_i$'s, \eqref{BKChernoff} holds if each $\Omega_i$ is linearly ordered.
\end{thm}
Unlike \eqref{BKChernoff}, this is neither stronger nor weaker than Lemma \ref{JUB} even when it appiles, because arbitrary independent events need not occur disjointly. For example, if $\Omega=\{0,1\}^n$ with uniform measure and, for $i \in [n-1]$, $A_i$ is the event that $\{\omega_i, \omega_n\}=\{0,1\}$, then $X \leq 1$ but $Z$ can be large.


\bn {\em \textbf{Historical Note.}} We learned of Lemma \ref{JUB} only after proving Theorem \ref{TBK'}; in fact our motivation for the theorem was to obtain something like the lemma, as in Remark (iii). Upon learning of the lemma, we realized its proof could be tweaked to give Theorem \ref{TBK2}.

\section{Proofs}

The proof of Theorem~\ref{TBK'}, which is similar to the original proof of \eqref{eq:BK} in \cite{BK},
is not hard but is a little awkward to write, and
a few additional definitions will be helpful. We prove it for increasing $A_i$'s; the decreasing case is of course analogous.

For $\gO=\prod_{i\in I}\gO_i$ and $S\sub I$, we take $\gO_S=\prod_{i\in S}\gO_i$
and, for $\go\in\gO$, $\go_S=(\go_i:i\in S)$.
For $A\sub \gO$ and $\go\in \gO_J$ for some $J\sub I$, $S\sub J$ is said
to {\em witness} $\go\in A$ if
$\omega' \in A$ whenever $\go'\in\gO$ and $\go'_S=\go_S$.
(This is of course abusive since we can't have $\go\in A$ unless $J=I$.)
We then (that is, for $\go\in \gO_J$) say
$A_1,\ldots,A_k $ ($\sub \gO$)
occur disjointly at $\go$ if there are disjoint $S_1\dots S_k\sub J$
such that $S_j$ witnesses $\go\in A_j$ $\forall j$ and,
for $\A = \{A_1,\ldots,A_k\}$, set
\[
X_\A(\go) =\max\{ |R| : R \subseteq [k], ~\mbox{the $A_j$'s indexed by $R$ occur disjointly at $\go$}\}.
\]
Thus the $X$ of Theorem~\ref{TBK'} 
is $X_\A$ evaluated at a random $\go\in \gO$.

\begin{proof}[Proof of Theorem~\ref{TBK'}]
Say $i\in [n]$ {\em affects} $A\sub \gO$ if there are $\go\in A$ and $\go'\in \gO\sm A$
with $\go_{[n]\sm\{i\}}=\go'_{[n]\sm\{i\}}$, and
for a collection $\B$ of events in $\gO$, let
$\psi(\B)$ be the number of $i\in [n]$ that affect at least two members of $\B$.

We 
proceed by induction on $\psi(\A)$.
If this number is zero then the laws of $X$ and $Y$ agree
(since the $A_j$'s are independent).
So we may assume $\psi(\A)\neq 0$, say (without loss of generality) the index 1 affects at least two of the $A_j$'s.

Let $(\gO_{n+j},\mu_{n+j})$, $j\in [k]$, be copies of $(\gO_1,\mu_1)$, independent of each
other and of $(\gO_1,\mu_1)\dots (\gO_n,\mu_n)$.
Let $(\gO^*,\mu^*) =\prod_{i=2}^{n+k} (\Omega_i,\mu_i)$
and (for $j\in [k]$)
\[
B_j=\{\go\in \gO^*: (\go_{n+j},\go_2\dots \go_n)\in A_j\}.
\]
Thus, apart from irrelevant variables, $B_j$ is a copy of $A_j$
gotten by replacing $(\gO_1,\mu_1)$ by $(\gO_{n+j}, \mu_{n+j})$.  In particular
$\Pr(B_j)=\Pr(A_j)$ and, with
$\B=\{B_1\dots B_k\}$, we have
$
\psi(\B) =\psi(\A)-1
$
(since $i\in [2,n]$ affects $B_j$ iff it affects $A_j$,
and $n+i$ affects $B_j$ iff $j=i$ and 1 affects $A_i$).
So by the inductive hypothesis it is enough to show
\beq{MP}
\mu(X_\A\geq r)~\leq ~\mu^*(X_\B\geq r)
\enq
for each positive integer $r$. Here it's convenient to work with the stronger conditional version:

\mn
\emph{Claim.}
{\em For each $y\in \gO_{[2,n]}$ (with $\mu_i(y_i)>0$ $\forall \, i\in [2,n]$),}
\beq{XZ}
\mu(X_\A(\go)\geq r\mid \go_{[2,n]}=y)~\leq ~\mu^*(X_\B(\go)\geq r\mid \go_{[2,n]}=y).
\enq

\mn
\emph{Proof of Claim.}
Since, for any $y\in \gO_{[2,n]}$ and $\go  \in \gO$ with $\go  _{[2,n]}=y$,
\[
\mbox{$X_\B(y) =X_\A(y)\leq X_\A(\go  )\leq X_\A(y)+1$,}
\]
we need only show \eqref{XZ} for $y$ with $X_\A(y)=r-1$
(since the left hand side of \eqref{XZ} is zero if $X_\A(y)\leq r-2$ and both sides are 1 if
$X_\A(y)\geq r$).

Given such a $y$, set
$
\f =\{x\in \gO_1:  X_\A(x,y) =r\}
$
and, for $i\in [k]$, let
$
\f_i\sub \gO_1$ consist of those
$x$'s for which there are $I\in \binom{[k]}{r}$ containing $i$ and disjoint $S_j$'s in $[n]$ ($j\in I$)
such that $S_j$ witnesses $(x,y)\in A_j$ (for $j\in I$) and $1\in S_i$.
Then, evidently,
\begin{itemize}
\item[$\circ$] each $\f_i$ is increasing,
\item[$\circ$] $\f=\cup_{i\in [k]}\f_i$,
\item[$\circ$] for $\go  \in \gO$ with $\go  _{[2,n]}=y$,
$X_\A=r$ iff $\go  _1\in \f$, and
\item[$\circ$] for $\go  \in \gO^*$ with $\go  _{[2,n]}=y$,
$X_\B\geq r$ iff $\go  _{n+j}\in \f_j$ for some $j\in [k]$,
\end{itemize}
whence
\begin{align*}
\mu(X_\A(\go)\geq r\mid \go_{[2,n]}=y) &= \mu_1(\f)
=1-\mu_1(\cap_{j\in [k]}\ov{\f}_j) \\
&\leq 1-\mbox{$\prod_{j\in [k]}\mu_1(\ov{\f}_j)$}
= \mu^*(X_\B(\go)\geq r\mid \go_{[2,n]}=y),
\end{align*}
where the inequality follows from that assumption that $\mu_1$ is PA.
\qedhere
\end{proof}

\bigskip For the proof of Theorem \ref{TBK2} we need just one little observation, which follows immediately from Reimer's Theorem \cite{Reimer} by induction: for events $\{A_i\}_{i \in I}$ in a product probability space with each factor linearly ordered,
\begin{align}
\label{eq:ReimerInduction}
\Pr(\square_{i \in I}A_i) ~ \leq ~ \prod_{i \in I} \Pr(A_i).
\end{align}

\begin{proof}[Proof of Theorem \ref{TBK2}] For some to-be-determined integer $r \leq k$ and each $I \subseteq [k]$ of size $r$, let $B_I$ be the indicator of $\square_{i \in I}A_i$. Let $\chi = r! \sum B_I$, so that \[ \E\chi = r!\sum_{|I|=r} \Pr(\square_{i \in I}A_i) \leq r! \sum_{|I|=r} \prod_{i \in I} \Pr(A_i) \leq \gl^r\] (by \eqref{eq:ReimerInduction}). 

The rest of the proof follows \cite[Lemma 2.46]{JLR} verbatim, so we will be brief. If $X \geq \lambda+t$ then $\chi \geq (\lambda+t)_r = \prod_{i=0}^{r-1}(\lambda+t-i)$, so by Markov, 
\[\Pr(X \geq \lambda+t) \leq \Pr(\chi \geq (\lambda+t)_r) \leq \frac{\lambda^r}{(\lambda+t)_r} = \prod_{i=0}^{r-1}\frac{\lambda}{\lambda+t-i}.\] 
Setting $r=t$ (to minimize the right hand side) yields
\[\log \Pr(X \geq \lambda+t) \leq \sum_{i=0}^{t-1} \log(\lambda/(\lambda+t-i)) \leq \int_0^t \log(\lambda/(\lambda+t-x)) \, \mathrm{d}x,\]
which, with calculus, gives the stronger bound in \eqref{BKChernoff}.
\end{proof}



\bibliographystyle{plain}
\bibliography{Allrefs.bib}

\begin{thebibliography}{1}

\bibitem{vdB}
J.~{van den} Berg.
\newblock Personal communication, Oct 2015.

\bibitem{BK}
J.~{van den} Berg and H.~Kesten.
\newblock Inequalities with applications to percolation and reliability.
\newblock {\em J. Appl. Probab.}, 22(3):556--569, Sept 1985.

\bibitem{Grimmett}
Geoffrey~R. Grimmett.
\newblock {\em Percolation}, volume 321 of {\em Grundlehren der mathematischen
  Wissenschaften}.
\newblock Springer-Verlag Berlin Heidelberg, Berlin, 2nd edition, 1999.

\bibitem{Harris}
T.~E. Harris.
\newblock A lower bound on the critical probability in a certain percolation
  process.
\newblock {\em Math. Proc. Cambridge Phil. Soc.}, 56(1):13--20, Jan 1960.

\bibitem{Janson}
Svante Janson.
\newblock Poisson approximation for large deviations.
\newblock {\em Random Structures Algorithms}, 1(2):221--229, June 1990.

\bibitem{JLR}
Svante Janson, Tomasz {\L}uczak, and Andrzej Ruci\'{n}ski.
\newblock {\em Random Graphs}.
\newblock Wiley-Interscience Series in Discrete Mathematics and Optimization.
  Wiley, New York, 2000.

\bibitem{Reimer}
David Reimer.
\newblock Proof of the {V}an den {B}erg--{K}esten conjecture.
\newblock {\em Combin. Probab. Comput.}, 9(1):27--32, Jan 2000.

\end{thebibliography}

\end{document}